\documentclass[12pt,oneside,leqno]{amsart}
\usepackage[all]{xy}
\usepackage{amsfonts,amsmath,oldgerm,amssymb,amscd}
\usepackage{enumerate}
\UseComputerModernTips
\numberwithin{equation}{section}
\usepackage[breaklinks]{hyperref}

\usepackage{float}

\makeatletter
\@namedef{subjclassname@2010}{%
	\textup{2010} Mathematics Subject Classification}
\makeatother

\usepackage{caption} 
\newtheorem{theorem}{Theorem}
\newtheorem{lemma}{Lemma}[section]
\newtheorem{prop}{Proposition}[section]
\newtheorem{coro}{Corollary}[section]

\theoremstyle{definition}
\newtheorem{defn}[theorem]{Definition} 
\newtheorem{conj}{Conjecture} 

\theoremstyle{remark}

\newcommand{\Z}{\mbox{$\mathbb Z$}}
\newcommand{\Q}{\mbox{$\mathbb Q$}}
\newcommand{\N}{\mbox{$\mathbb N$}}     

\title{Stability of Certain Higher Degree Polynomials}
\author{Shanta Laishram, Ritumoni Sarma} 
\address{Stat Math Unit, Indian Statistical Institute, New Delhi 110016, India}
\email{shanta@isid.ac.in}

\address{
	Department of Mathematics, Indian Institute of Technology Delhi, New Delhi-110016, India}
\email{ritumoni@gmail.com }
\author{Himanshu Sharma}
\address{ Department of Mathematics, Indian Institute of Technology Delhi, New Delhi-110016, India}\email{himanshusharma985@gmail.com}

\usepackage{fancyhdr}
\begin{document}
	
	\date{\today}
	\maketitle
	
	\begin{abstract}
	One of the interesting problems in arithmetic dynamics is to study the stability of polynomials over a field. In this paper, we study the stability of $f(z)=z^d+\frac{1}{c}$ for $d\geq 2$, $c\in{\mathbb{Z}\setminus\{0\}}$. 
	We show that for infinite families of $d\geq 3$, whenever $f(z)$ is irreducible, all its iterates are irreducible, that is, $f(z)$ is stable.  For $c\equiv 1\pmod{4}$, we show that all the 
	 iterates of $z^2+\frac{1}{c}$ are irreducible. Also we show that for $d=3$, if $f(z)$ is reducible, then the number of irreducible factors of each iterate of $f(z)$ is exactly $2$ for $|c|\leq{10^{12}}$. 
	 	\end{abstract}

	\section{Introduction}
	An important question in the field of arithmetic dynamics is to study the recurrence sequences satisfying $t_{n}=f(t_{n-1})$, where $t_{0}\in{\mathbb{Q}}$ and $f(z)\in{\mathbb{Q}[z]}$ with $\deg(f(z))\geq{2}$. One can ask how many primes are there in the sequence $(t_{n})$ or which primes are dividing at least one element of the sequence $(t_{n}).$ Many authors have investigated these questions in [\cite{rice2007},\cite{krieger2013},\cite{ingram2009},\cite{gratton2013},\cite{faber2011}]. One interesting problem in this direction is about the stability and eventual stability of polynomials over a field. 
 In fact stability and eventual stability have been recently used in proving finite index results for some arboreal representations in \cite{2021finite} and \cite{2019finite}. 
 
 If each iterate of $f(z)\in{\mathbb{Q}[z]}$ is irreducible over $\mathbb{Q}$ then we say that $f(z)$ is {\it stable} over $\mathbb{Q}$. More generally if the number of irreducible factors of iterates of $f(z)$ is bounded by a constant, that is, there exist $n_0\in{\mathbb{N}}$ such that the number of irreducible factors of $f^n(z)$ remains constant for $n\geq{n_0}$ then we say that $f(z)$ is 
  {\it eventually stable}. 
   
   We consider the stability and  eventual stability of  $z^d+b$ with $b\in \Q, b\neq 0$.  Put $b=\frac{a}{c}$ with $a, c\in \Z, c\neq 0$. It was shown in \cite[Theorem $1.6$]{hamblen} that $f(z)=z^d+\frac{a}{c}$ is eventually stable when $a\ne{1}$. The stability and eventual stability of $z^d+\frac{1}{c}$ is not  known completely. 
   Even for the quadratic polynomial $z^2+\frac{1}{c}$ over $\mathbb{Q}$, it is not completely known though some partial 
   results are available in \cite{jones2020}. We refer to \cite{benedetto} for a more detailed survey.  
   In this paper, we consider the stability and eventual stability of the polynomial $z^d+\frac{1}{c}, c\in \mathbb{Z}\setminus\{0\}$. We prove the following results.  
     
	\begin{theorem}\label{main1}
	Let $d\geq 2$ be a positive integer and $f(z)=z^d +\frac{1}{c}$  where $c\ne{0}$ is an integer. We have 
	\begin{enumerate}[(a)]
	\item  Let $d=2$. If $c\equiv 1\pmod{4}$ then each iterate of $f(z)$ is irreducible over $\mathbb{Q}$. 
	\item   Let $d>2$. Assume that $f(z)$ is irreducible. Then $f(z)$ is stable over $\mathbb{Q}$ when 
		\begin{enumerate}[$1.$]
		    \item   $d\geq3$ is odd.
		    \item   $d=2^r$ with $r\geq 2.$
		    \item  $d=2^r\cdot 3^s$ with $r, s\geq{1}.$
		     \item  $d=2^r\cdot 5^s\cdot 7^t$ with $r\geq 1, s, t\geq 0$ and $d\equiv 1\pmod{3}$
		    \item  $d\equiv 4\pmod{12}$.
		    \end{enumerate}
		    \end{enumerate}
	\end{theorem}
We note that when $d=2$ and $c=-m^2$, then $c\equiv 0, 3\pmod{4}$ and  $f(z)=z^2 +\frac{1}{c}$ is reducible.  We refer to \cite[Theorem 1.3]{jones2020} for other values of $c$ for which 
each iterate of $z^2+\frac{1}{c}$ is irreducible over $\mathbb{Q}$. The case $d>3$ odd and not divisible by $3$ follows from \cite[Theorem 7]{danielson2002}. From Theorem \ref{main1} and the remarks made in \cite{danielson2002} and , we believe that the following conjecture is true. 
\begin{conj}\label{stable}
	If $f(z)=z^d+\frac{1}{c},c\in{\mathbb{Z}\setminus\{0\}}, d\geq{3}$, is irreducible over $\mathbb{Q}$ then $f(z)$ is 
	stable over $\mathbb{Q}$.    
	\end{conj}
Though we are not able to completely prove this conjecture unconditionally, by using an explicit version of $abc$-conjecture 
due to Baker \cite{baker2004} (see Conjecture \ref{exp-abc}), we are able to show that the Conjecture \ref{stable} is true. 

	\begin{theorem}\label{main2}
	The explicit $abc-$Conjecture implies Conjecture \ref{stable}. That is $f(z)=z^d+\frac{1}{c}$ with  $c\in{\mathbb{Z}\setminus\{0\}}$ and $d\geq 3$ is stable over $\mathbb{Q}$ whenever $f(z)$ is irreducible over $\Q$. 
	    \end{theorem}
When $f(z)=z^d+\frac{1}{c}$ reducible over $\Q$, nothing much is known about the irreducible factors of iterates of $f(z)$ 
for  $d\geq 3$. In this paper, we prove the following result about the eventual stability of $f(z)=z^3+\frac{1}{c}$ when it is reducible.  
		\begin{theorem}\label{main3} 
		Let $f(z)=z^3+\frac{1}{c}, c\in{\mathbb{Z}\setminus\{0\}}$. If $f(z)$ is reducible over $\mathbb{Q}$ then 
		$f^n(z)$ has exactly two irreducible factors over $\mathbb{Q}$ for each $n\in{\mathbb{N}}$ and for $c$ 
		with $|c|\leq{10^{12}}.$
	\end{theorem}
	We believe that above result should be true for all $c$ and we propose the following conjecture. 
\begin{conj}\label{d=3-factors}
	Let $f(z)=z^3+\frac{1}{c}, c\in{\mathbb{Z}\setminus\{0\}}$. If $f(z)$ is reducible over $\mathbb{Q}$ then 
	$f^n(z)$ has exactly two irreducible factors over $\mathbb{Q}$ for each $n\in{\mathbb{N}}$ and for all $c$.    
	\end{conj}	
	
	We prove Theorems \ref{main1} and \ref{main2} in Section 4.  Theorem \ref{main3} is proved in Section 3. In Section 2, we give the preliminaries.

	\section{Preliminaries}
	Let $K$ be a number field, $\phi(z)\in K(z)$ and let $\alpha\in \mathbb{P}^1(K).$ 
	Assume that $\phi(z)=\frac{f(z)}{g(z)}$, where  $f(z)$ and $g(z)$ are coprime in $K[z]$. The degree of a 
	rational function $\phi(z)=\frac{f(z)}{g(z)}$ is defined as $\deg{\phi(z)}=\max\{\deg{f(z)},\deg{g(z)}\}$. 
	Let $\phi^n(z)=\phi(\phi^{n-1}(z))$, the $n$-th iterate of $\phi(z)$ and put $\phi^n(z)=\frac{f_{n}(z)}{g_{n}(z)}$, 
	where $f_{n}(z), g_{n}(z)\in K[z]$ are coprime  polynomials. 
\begin{defn}The pair $(\phi,\alpha)$ is said to be stable over $K$ if $f_{n}(z)-\alpha g_{n}(z)$ is irreducible over $K$ for each $n\in{\mathbb{N}}$.
    \end{defn} 
However stability is not preserved under field extensions. A weaker condition called eventual stability behaves well with respect to finite field extensions. 
\begin{defn} If there exist a constant $C(\phi,\alpha)$ such that the number of irreducible factors of $f_{n}(z)-\alpha g_{n}(z)$, for all $n\geq{1}$, is bounded by $C(\phi,\alpha)$ then we say that ($\phi$,\,$\alpha$) is {\it eventually stable} over $K$. If the number of irreducible factors of $g_{n}(z)$ is similarly bounded then we say that $(\phi,\infty)$ is eventually stable.
\end{defn}  
 If $(\phi,0)$ is eventually stable then we say that $\phi(z)$ is eventually stable. Eventual stability of a rational 
function over $K$ is preserved over finite extensions of $K$.
	
	A map $\nu :K\to \mathbb{Z} \cup \{\infty\}$ is called a {\it discrete} valuation on $K$ if it satisfies the following properties:
	\begin{itemize}
		\item [(1)]  $\nu(x)=\infty \Leftrightarrow x=0$,
		\item [(2)]  $\nu(xy)=\nu(x)+\nu(y)~~  \forall\, x,~y \in K$ and
		\item [(3)]  $\nu(x+y)\geq \inf\{\nu(x),\nu(y)\}$.
	\end{itemize}
	
	Suppose $p$ is a rational prime. Then every rational number $x$ can be written as $p^t x_{0}$ where $t\in\mathbb{Z}$ and $p$ divides neither the numerator nor the denominator of $x_{0}$. Then, $\nu_{p}:\mathbb{Q}\setminus\{0\}\to \mathbb{Z}$ defined by  $\nu_{p}(x)=t$ is a discrete valuation which is called the {\it $p$-adic valuation} on $\mathbb{Q}$.
	 
 	Given a discrete valuation $\nu$ on the number field $K$, $R$ denote the following ring $\{x\in K:\nu(x)\geq{0}\}$. Then $\mathfrak{p}=\{x\in K:\nu(x)>0\}$ is the unique maximal ideal of the ring $R$. The field $k=R/\mathfrak{p}$ is called the residue field of $\nu$.
	Denote by $\tilde{x} \in {\mathbb{P}^1(\kappa)}$ the reduction modulo $\mathfrak{p}$ of $x\in \mathbb{P}^1(K)$. Further we  denote by $\tilde{f}(z)$ the polynomial obtained from $f(z)\in R[z]$ by reducing each coefficient modulo $\mathfrak{p}$.
	
	Let $\phi(z)=\frac{f(z)}{g(z)}\in{K(z)}$ where $f(z),g(z)\in{R[z]}$. Then we say that $\frac{f(z)}{g(z)}$ is {\it normalized} if $f(z),g(z)\in{R[z]}$ are coprime and atleast one of the coefficients of $f(z)$ or $g(z)$ is a unit in $R$.
	  Define $\tilde{\phi}(z)=\frac{\tilde{f}(z)}{\tilde{g}(z)}$. Suppose $\phi(z)$ is non constant. Then  $\phi(z) \in {K(z)}$ is said to have {\it good reduction} at $\nu$ if $\deg(\tilde{\phi}(z))=\deg(\phi(z))$.
	We say that $\phi(z)$ is {\it bijective on residue extensions} for the discrete valuation $\nu$ on $K$ if $\tilde{\phi}(z)$ defines a bijection on $\mathbb{P}^1(E)$ for every finite extension $E$ of the residue field $\kappa$ of $\nu$.
	\begin{prop}\label{1} \cite{jones2017}
		Let $K$ be a number field with discrete valuation $\nu$, $\phi(z)\in{K(z)}$ be such that $\deg{\phi(z)}\geq{2}$ and let $\alpha\in{\mathbb{P}^{1}(K)}$ be non-periodic with respect to $\phi(z)$. Assume $\phi(z)$ is bijective on residue extensions for $\nu$ and has good reduction at $\nu$. Assume further $\phi^n(z) =\frac{f_{n}(z)}{g_{n}(z)}$ is normalized and $\alpha\ne{\infty}$. Then the number of  irreducible factors of $f_{n}(z)-\alpha g_{n}(z)$ over $K$ is at most
		\begin{itemize}
				\item [1.]  $\nu(\phi(\alpha)^{-1} - \alpha^{-1})$  if $\nu(\alpha)<0$,
			   \item [2.]  $\nu(\phi^i(\alpha)-\alpha)$ if $\nu(\alpha)\geq{0}$ for $ i = \min\{{n\geq{1}:\tilde{\phi}^n (\tilde{\alpha})=\tilde{\alpha}}\}.$
		 		\end{itemize}
		 
	\end{prop}
	\begin{lemma} \cite{jones2020}\label{2} Suppose $g(x)\in{K[x]}$ is a monic irreducible polynomial with degree $d\geq{1}$ and $char(K)\neq 2$. Let $f(x)\in K[x]$ be a monic quadratic polynomial and $\gamma $ be 
	such that $f'(\gamma)=0$. If  no element of $$\{(-1)^d g(f(\gamma))\}\cup\{g(f^{n}(\gamma)):n\geq{2}\}$$ is a square in K, then $g(f^{n}(x))$ is irreducible over $K$ for each $n\in{\mathbb{N}}$.
	\end{lemma}
We will use the following result for $d=3$ in the proof of Theorem \ref{main3} which deals with the irreducible factors of the iterates of $z^3+\frac{1}{c}$ when it is reducible.  
	\begin{lemma}[Capelli's Lemma, \cite{jones2008}]\label{Capelli}
		Suppose $f(x), g(x)$ are polynomials over the field $K$ such that $g(x)$ is irreducible. Then $g(f(x))$ is irreducible over $K$ if and only if $f(x)-\beta$ is irreducible over $K(\beta)$ for every root $\beta\in{\overline{K}}$ of $g(x)$.
	\end{lemma}

	\begin{prop}\cite{danielson2002}\label{3}
	    Let $n\geq{3}$ be an odd integer and let $f(x) = x^n-b\in{\Q[x]}$. For each $m\geq{1}$, let $S(n,m)=\{b\in{\Q}: f^{m}\, \text{is irreducible but} \,f^{m+1}\, \text{is reducible}\, \text{over}\,\,{\Q}\}.$ Then $S(n)=\cup_{m=1}^{\infty}S(n,m)$ is finite and is empty if $3\nmid n$. 
	    \end{prop}
	    The above result was proved by using Lemma \ref{Capelli} and the non-existence of primitive solutions of the  
	    generalized Fermat equation. Let $p, q, r$ be integers $\geq 2$ and consider the  generalized Fermat equation
	\begin{equation} \label{eqn}
		x^p+y^q={z^r}.
	\end{equation}
	Given a triple $(a,b,c)\in{\mathbb{Z}}^3$, we say that $(a,b,c)$ is a solution of equation $(2.1)$ if $a^p+b^q=c^r$. If 
	$a,b,c$ are pairwise coprime then this solution is called ${\it proper}$. A proper solution $(a, b, c) $ is \emph{primitive}  if $abc\neq 0$.  A well-known  conjecture regarding the generalized Fermat equations is due to Tijdeman and Zagier (see \cite{laishram2011}), also known as Beal's conjecture.  
	
	\begin{conj}\label{Beal}
There are no primitive solutions of the diophantine equation $$x^p+y^q=z^r$$ in ${\mathbb{Z}}$  for $p, q, r\geq 3.$  
\end{conj}	
This is open. However there are a number of partial results on this conjecture. For the proof of our theorems, we need the following result on the non-existence of primitive solutions of the following equations, see \cite{bennett2015}.
\begin{lemma}\label{b}
There are no primitive solutions for the equation $x^p+y^q=z^r$ when 
	$$(p, q, r)\in \{(2, 3, 7), (2, 3, 10), (2, n, 4), (2, n, 6), (3, 3, 2n), (3, n, 6), (n, n, 2), (n, n, 3)\}$$
	where $n\geq 2$ and further $n\geq 3$ when $(p, q, r)\in\{(3, n, 6),(n,n,3)\}$; $n\geq{4}$ when 
	$(p,q,r)\in\{(2,n,6),(n,n,2)\}$ and $n\geq{6}$ when $(p,q,r)=(2,n,4)$.
\end{lemma}
	The next result is a Catalan's conjecture, now a theorem of Mih$\breve{a}$ilescu \cite{mihai}. 
	\begin{lemma}\label{a}
	The only solution of $x^m-y^n=1$ in integers $x, y, m>1, n>1$ with $xy\neq 0$ is given $3^2-2^3=1$.   
		\end{lemma}

We end with this section with the explicit abc-conjecture due to Baker \cite{baker2004}. If $m$ is a positive integer, then its {\it radical} $N(m)$ is the product of 
distinct prime divisors of $m$ and $\omega(m)$ denotes the number of distinct primes dividing  $m$.  
\begin{conj}\label{exp-abc}[Explicit abc-conjecture]
Let $a,b$ and $c$ be pairwise coprime integers satisfying $a+b=c.$ Then $$c<\frac{6}{5}N\frac{(\log N)^\omega}{\omega!}$$ 
where $N=N(abc)$ and $\omega=\omega(N).$ 
\end{conj}
An easily applicable formulation was given by Laishram and Shorey \cite[Theorem 1]{laishram2011}. The next result is contained in  \cite[Theorem 1]{laishram2011}. 
\begin{prop}\label{abc-LS}
Assume Conjecture \ref{exp-abc}. Let $a,b$ and $c$ be pairwise coprime integers such that $a+b=c.$ Then  
$$c<N^{1+\frac{3}{4}} \qquad {\rm where} \quad N=N(abc).$$ 
Further for $0<\epsilon\leq{\frac{3}{4}}$, there exist $N_{\epsilon}$, depending only on $\epsilon$, such that whenever  $N\geq{N_{\epsilon}}$, we have 
$c<N^{1+\epsilon}$.  In particular, for $\epsilon=\frac{7}{12}$, $N_{\epsilon}=\exp(204.75)$. 
\end{prop}
Explicit $abc-$Conjecture can be used to give a general result on the Conjecture \ref{Beal}, see \cite[Theorem 3]{laishram2011}. 
\begin{prop}\label{abc-Beal}
Assume Conjecture \ref{exp-abc}. Then there are no primitive solutions of the equation $x^p+y^q=z^r$ with $p\geq 3, q\geq 3, r\geq 3$ and 
$$[p, q, r]\notin \{[3, 5, \ell]: 7\leq \ell\leq 23, \ell \ {\rm prime}\}\cup \{[3, 4, \ell]: \ell \ {\rm prime}\}$$
where $[p , q, r]$ denote all the permutations of ordered triples $(p ,q, r)$.  
\end{prop}

\section{Proof of Theorem \ref{main3}}

We start this section with a lemma which is analogous to Lemma \ref{2}.
	\begin{lemma}\label{4}
		Let $g(z)$ be a monic irreducible polynomial over a number field $K$ and let $f(z)=z^3 +\frac{1}{c} \in{K[z]}$. If none of $\Big\{g(f^n(0))\Big\}_{n\geq{1}}$ is a cube in $K$ then 
		$g(f^n(z))$ is irreducible for each $n\in{\N}$.
		\end{lemma} 
		\begin{proof}
			We prove it by induction on $n$. For $n=0,~~g(z)$ is given to be irreducible. 
			Assume now irreducibility of $g(f^{n-1}(z))$. Then by Capelli's Lemma \ref{Capelli},  $g(f^{n}(z))$ is irreducible over $K$ if and only if 
			$z^3+1/c-\gamma$ is irreducible over $K(\gamma)$, that is, $1/c-\gamma$ is not a cube in $K(\gamma)$, for every root $\gamma$ of $g(f^{n-1})(z)$. Note that
\begin{align*}
N_{K(\gamma)/K}\left(\frac{1}{c}-\gamma\right)&=\prod_{\alpha~~ \text{is root of}~~ g(f^{n-1}(z))}\left(\frac{1}{c}-\alpha\right)=g\left(f^{n-1}\left(\frac{1}{c}\right)\right)=g(f^{(n)}(0)).
\end{align*}
If $N_{K(\gamma)/K}(1/c-\gamma)$ is not a cube in $K$ then $1/c-\gamma$ is not a cube in $K(\gamma)$. Thus, if $g(f^{(n)}(0))$ is not a cube in $K$ then $z^3+1/c-\gamma$ is 
irreducible over $K(\gamma)$ so that $g(f^{n}(z))$ is irreducible over $K$ for each $n\in{\N}.$
\end{proof}

	\begin{lemma}
		Let $f(z)=z^d+\frac{1}{c}$ with $d$ an odd integer. For any prime $p$ dividing $d$, if $f^n(0)$ is not a $p$-th power in $\mathbb{Q}$ for $c>0$ then it is also not a $p$-th power for $c<0$ 
		in $\mathbb{Q}$. Additionally, If $d=3$ then similar result holds for $g(f^n(0))$ for any divisor $g(z)$ of $f(z)$.
	\end{lemma}
	\begin{proof} Let $f_{c}(z)=z^d+\frac{1}{c}$ with $d$ an odd integer. By induction, we have $f^n_{-c}(0)=-f^n_{c}(0)$ for all $n\geq{1}$. Therefore $f^n_{-c}(0)$ is a $p$-th power if and only if $f^n_{c}(0)$ is a $p$-th power in $\mathbb{Q}$. Suppose $d=3$, $f_{c}(z)$ is reducible and $g(z)|f(z)$. It is also easy to observe that $g_{-c}(f^n_{-c}(0))=\pm{g_c(f^n_{c}(0))}$. Hence the proposition follows. 
	\end{proof}

	Let $f(z)=z^3+\frac{1}{c}$. We know that $f(z)$ is reducible if and only if $c=m^3$ for some $m\in{\mathbb{Z}}.$ In that case $f(z)=(z+\frac{1}{m})(z^2-\frac{z}{m}+\frac{1}{m^2})$. Set
	\begin{equation} \label{g1g2}
		g_{1}(z)=z+\frac{1}{m} \quad  {\rm and} \quad g_{2}(z)=z^2-\frac{z}{m}+\frac{1}{m^2}.
	\end{equation}	
		 Since $g_{1}(f(z))=z^3+1/m^3+1/m=z^3+\frac{m^2+1}{m^3}$, $g_{1}(f(z))$ is reducible if and only if $m^2+1$ is a cube in $\mathbb{Z}.$ 
		 By Lemma \ref{a}, $m^2+1$ is not a cube unless $m=0$. Hence $g_{1}(f(z))$ is irreducible for all $m\ne{0}.$  For considering the irreducible factors of iterates of $f(z)$, we require the following 
		 lemma for which the proof is similar to that of \cite[Proposition 5.4]{jones2008}.
		 \begin{lemma}\label{rigid divisible}
		Let $f(z)=z^3+\frac{1}{c}$ be reducible with the irreducible factors $g_1(z)$ and $g_2(z)$ as in \eqref{g1g2}. Then $(g_i(f^n(0))),n\geq{1},$ is a rigid divisibility sequence for $i=1, 2$.
		 \end{lemma}

	\subsection{Proof of Theorem \ref{main3}:} 
 Let $f(z)=z^3+\frac{1}{c}$ be reducible. Then $c=m^3$ for some $0\neq m\in{\mathbb{Z}}$ and write $f(z)=g_{1}(z)g_{2}(z)$, where $g_{1}(z),g_{2}(z)$ are given by \eqref{g1g2}. 
Let $w_{n}(m)$ be the numerator of $g_{1}(f^{n-1}(0))$, so that  $ w_{2}(m)$ is the numerator of $g_{1}(f(0))$. 
We observe that $w_{2}(m)=m^2+1$ is not a cube in $\mathbb{Z}$ by Lemma \ref{a}. By Lemma \ref{rigid divisible}, it follows that 
$({w_{n}(m)})$ is a rigid divisibility sequence. Hence $w_{2n}(m)$ is not a cube in $\mathbb{Z}$ for each $n\geq{1}$.
		
		 For a given  $k\in{\mathbb{N}}$ and $m\not\equiv 0\pmod{k}$, the sequence $(g_{1}(f^n(0))\pmod{k}$ eventually becomes a repeating cycle and we search for values of $k$ and congruence classes of $m$ modulo $k$ such that $(g_{1}(f^n(0)) \textrm{mod}\ k)$ is not a cube for each $n\in{\N}$. Since we have shown above that $w_{2j}$ is not a cube in $\mathbb{Z}$ for each $j\geq{1}$, that is, $g_{1}(f^{2j-1}(0)$ is not a cube in $\Q$, it is enough to check that $(g_{1}(f^{2j}(0)) \textrm{mod}\ k)$ is not a cube for $j\geq{1}$. In fact, with the help of SAGE we verify  for each $j\in{\N}$ that $(g_{1}(f^{2j}(0))$ is not a cube for $m$ belonging to the congruences classes  modulo $k\in \{7, 13, 19, 31, 37, 43\}$ that are listed in Table 1.
			\begin{table}[H]
			\begin{tabular}{|l|c|}\hline 
			$k$ & $m \pmod{k}$\\ \hline 
			$7$ & $m\equiv \pm1, \pm3$\\ \hline 
			$13$ & $m\equiv \pm1,\pm2,\pm3,\pm6$ \\ \hline 
			$19$ & 	$m\equiv \pm2, \pm4$  \\ \hline 
			$31$ &  $m\equiv \pm1,\pm3,\pm4,\pm6,\pm8,\pm9,\pm10,\pm11,\pm12$  \\ \hline 
			$37$ & 	$m\equiv \pm3,\pm9,\pm17$ \\ \hline 
			$43$ & $m\equiv \pm2,\pm5,\pm8,\pm10,\pm12,\pm13,\pm14,\pm15,\pm20$\\ \hline 
				\end{tabular}
			\caption{\label{tab:Table 1}}
					\end{table}
			One can verify that the congruence classes in Table 1 cover all integers belonging to the interval $[1,10^4]$ except for 267 of them. For each of the remaining integers, we could find a prime $p<150$ such that  $g_{1}(f^n(0)\pmod{p}$  is not a cube for every even integer $n$. For example, $p=73$ works for $m=4342$. Hence by  Lemma \ref{4},  $g_{1}(f^n(z)$ is irreducible for all $n\geq{1}$  when $1\leq{m}\leq{10^4}$.
			
			Let $x_{n}(m)$ be the numerator of  $g_{2}(f^{n-1}(0))$. In particular $x_{2}(m)$ is the numerator of $g_{2}(f(0)).$ It again follows from Lemma \ref{rigid divisible} that $(x_{n}(m))$ is a rigid divisibility 
			sequence. Since $g_{2}(f(0))=\frac{m^4-m^2+1}{m^6}$,~~we have $x_{2}(m)=m^4-m^2+1$. Now $x_2(m)$ is a cube if and only if the elliptic curve $y^2-y+1=x^3$ has integral point with $y=m^2$. It follows from the curve $243.a1$ in LMFDB\,\cite{collaboration2013functions} that $y^2-y+1=x^3$ has only integral points $(1,1), (7,19),(1,0),(7,-18)$. Since $m\ne{0}$, $m^4-m^2+1$ is a cube in $\mathbb{Q}$ for $m=\pm{1}$ only, that is, $x_{2}(m)$ is not a cube in $\mathbb{Z}$ for $|m|\geq{2}$. Let $m=1$. In this case $f(z)=z^3+1$ and $g_{2}(z)=z^2-z+1$, we verify that   $z^6+z^3+1=g_{2}(f(z))$ is irreducible over $\mathbb{Q}$. Further we checked 
			that $g_{2}(f^n(0))\pmod{7}$ is not a cube for each $n\geq{2}$. Hence by Lemma \ref{4}, $g_{2}(f^n(z))$ is irreducible for all $n\in{\N}.$ Similarly we verified that $g_2(f^n(z))$ is irreducible for each $n\in{\N}$ 
			when $m=-1$. Hence we take $|m|\geq 2$. Since $(x_{n}(m))$ is a rigid divisibility sequence and $x_2(m)$ is not a cube, for $n\geq 1$, $x_{2n}(m)$ is also not a cube in $\mathbb{Z}.$ Again with the help of SAGE we verify  for each $j\in{\N}$ that $(g_{1}(f^{2j}(0))$ is not a cube for $m$ belonging to the congruences classes  modulo $k\in \{7, 13, 19, 31, 37\}$ that are listed in Table 2.
			\begin{table}[H]
			\begin{tabular}{|l|c|}\hline
			$k$ & $m \pmod{k}$\\ \hline 
			$7$ & $m\equiv \pm1,\pm2,\pm3$ \\ \hline
			$13$ &  $m\equiv \pm1,\pm2,\pm3,\pm4,\pm6$\\ \hline
			$19$ & $m\equiv \pm3,\pm5$ \\ \hline
			$31$ & $m\equiv \pm1,\pm4,\pm7,\pm8,\pm9,\pm11,\pm14$ \\ \hline
			$37$ & $m\equiv \pm4,\pm7,\pm9,\pm12,\pm16,\pm17,\pm18$ \\ \hline	
			\end{tabular}
		\caption{\label{tab:Table 2}}
			\end{table}
	One can verify that the congruence classes in Table 2. cover all integers belonging to the interval $[1,10^4]$ except for $88$ of them. For each of the remaining integers, we could find a prime $p<150$ such that  $g_{1}(f^n(0)\pmod{p}$  is not a cube for every even integer $n$. For example, $p=67$ works for $m=2730$. Hence $(g_{2}(f^n(z))$ is irreducible for each $n\in{\N}$ by Lemma \ref{4}. This proves Theorem \ref{main3}.

\section{Proof of Theorems  \ref{main1} and \ref{main2}}

For the proof of Theorems  \ref{main1} and \ref{main2}, we need the following lemma. 

	\begin{lemma}\label{7} Let $g(z)$ be a monic irreducible polynomial over a number field $K$ and let $f(z)=z^d +\frac{1}{c} \in{K[z]}$. Then for $m\geq 1$, $g(f^m(z))$ is irreducible  if  $g(f^{m-1}(z))$ is irreducible 
	and 	${g(f^{m}(0))}$ is not a $p$-th power in $K$ for each prime $p$ dividing $d$.  Hence for $m\geq 1$, $g(f^m(z))$ is irreducible  if  ${g(f^{j}(0))}$ is not a $p$-th power in $K$ for each prime $p$ dividing $d$ and for each $1\leq j\leq m$. 
	\end{lemma} 
	\begin{proof}
	Assume $g(f^{m-1}(z))$ be irreducible. By Lemma \ref{Capelli} we have $g(f^{m}(z))$ is irreducible if and only if for every root $\beta$ of $g(f^{m-1}(z))$, $f(z)-\beta=z^d+1/c-\beta$ is irreducible over $K({\beta)}$. By \cite[Theorem $9.1$]{lang2002}, $f(z)-\beta$ is irreducible if for every prime $p$ dividing $d$ we have $\beta-1/c\notin{K(\beta)}^p$ and if $4|d$ then $\beta-1/c\notin{-4K(\beta)}^4$. We know that
			$\beta-\frac{1}{c}\notin{K(\beta)}^p$ for a prime $p$ if  $N_{K(\beta)/K}(\beta-\frac{1}{c})\notin{K}^p$. Note that
			\begin{align*}
				N_{K(\beta)/K}\left(\beta-\frac{1}{c}\right)&=\prod_{\text{roots}~\alpha \ {\rm of} \  g(f^{m-1}(z))} \left(\alpha-\frac{1}{c}\right)
				=(-1)^tg(\left(f^{m-1}\left(\frac{1}{c}\right)\right)=(-1)^tg(f^{(m)}(0))
			\end{align*}
			where $t=\deg\,(g(f^{m-1}(z)))$. Hence if $g(f^{(m)}(0))$ is not a $p$-th power in $K$ then $\beta-1/c\notin{K(\beta)}^p$. Now let $4|d.$ Since $g(f^{(m)}(0))$ is not a square in $K$ and $4|t$, we have 
			from  $N_{K(\beta)/K}(1/c-\beta)=g(f^{(m)}(0))$ that  $1/c-\beta$ is not a square and hence $1/c-\beta \notin{4K(\beta)}^4$. 	Hence we conclude that if $(g(f^{m}))(0)$ is not a $p$-th power in $K$  for each prime $p$ dividing $d$ and 
			$g(f^{m-1}(z))$ is irreducible, then $g(f^m(z))$ is irreducible. This proves the first assertion. 
			The latter follows inductively from the  first assertion.   
		\end{proof}

Apply Lemma \ref{7} to $K=\Q$ and $f(z)=g(z)=z^d+\frac{1}{c}$ whenever it is irreducible over $\Q$.  Let $a_{n}$ denote the numerator of $f^{n}(0)$. Then $a_{n}$ satisfies $a_1=1$ and for $n>1$, 
$$a_{n}=a_{n-1}^{d}+c^{d^{n-1}-1} \qquad  {\rm and} \qquad f^{n}(0)=\frac{a_{n}}{c^{d^{n-1}}}.$$ 
By induction,  $a_{n},a_{n+1}$ and $c$ are pairwise coprime for $n\in{\mathbb{N}}$. Note that for any prime $p$ dividing $d$, the denominator of $f^{n}(0)$ is always of the form $(c^e)^p$ for some positive integer $e$. So, for a prime $p$ that divides $d$, $f^{n}(0)$ is not a $p$-th power in $\mathbb{Q}$ unless $a_{n}$ is a $p$-th power in $\mathbb{Z}$. Let $n=2$. Then $a_1=1$ implying $a_2=1+c^{d-1}$. By Lemma \ref{a}, $a_2$ is a power only when $c=2, d=4$ in which case $a_2=3^2$.  By Lemma \ref{7},  $f^2(z)$ is irreducible for all 
$(d, c)\neq (4, 2)$. For $(d, c)=(4, 2)$, we check that $f^2(z)=(x^4+\frac{1}{2})^4+\frac{1}{2}$ is irreducible over $\mathbb{Q}$. Thus to prove Theorem 1 it is enough to consider $n\geq 3$. 
  
\subsection{{\bf Proof of Theorem \ref{main1}:} } Let $f(z)=z^d+b$. As in the statement of Proposition \ref{3}, 
let $S(d, m)$ be the set of $b\in{\mathbb{Q}}$ for which $f^m(z)$ is irreducible but $f^{m+1}(z)$ is not and set $S(d)=\cup_{m=1}^{\infty}S(d, m)$. By  
Proposition \ref{3}, $S(d)=\emptyset$ for all $d$ not divisible by $2$ or $3$. In other words, for $d$ not divisible by $2$ or $3$, if $f(z)=z^d+\frac{1}{c}$ is irreducible then all its iterates are also irreducible.  Hence for the proof of Theorem \ref{main1}, we may suppose that $2|d$ or $3|d$. 
	
Let $d=2$. We apply Proposition \ref{1} to $K=\mathbb{Q}$, $\phi(z)=f(z)=z^2+\dfrac{1}{c}$ and $\nu=\nu_2,$ the $2$-adic valuation on $\mathbb{Q}$. Here the 
corresponding residue field is $\mathbb{F}_2$. $f$ has good reduction when $c$ is odd in which case $\tilde{f}(z)=z^2+1$. Let $c$ be odd. Then $\tilde{f}(0)=1, \tilde{f}(1)=0$ so that 
the period $i=2$ by Proposition \ref{1}. That is $f^n(z)$ has at most $\nu_2(f^2(0))$ irreducible factors. Now $f(0)=\frac{1}{c}, f^2(0)=\frac{1+c}{c^2}$ so that $\nu_{2}(f^2(0))=\nu_{2}(1+c)$. When $c\equiv 1\pmod{4}$, 
we have $\nu_{2}(1+c)=1$ so that all the iterates $f^n(z)$ are irreducible over $\mathbb{Q}$. This proves $(a)$. 

Now suppose $d\geq 3$. As stated after the proof of Lemma \ref{7}, we need to show for $n\geq 3$ that for no prime $p$ dividing $d$ $a_{n}$ is a $p$-th power in $\mathbb{Z}$ . We consider various cases of (b) one by one. 

Let $d$ be an odd integer divisible by $3$. Then $d^{n-1}-1$ is even and $a_{n}=a_{n-1}^{d}+c^{d^{n-1}-1}=x_1^3+y_1^{2k}$ and if $p\ne3$ is a prime dividing $d$ then we also have $a_{n}=a_{n-1}^{d}+c^{d^{n-1}-1}=x_2^p+y_2^2$ for some integers $x_1,x_2, y_1,y_2,k$ with $k\geq 3$. For each $n\geq{3}$, $a_n$ is not a $p$-th power for any prime divisor of $d$ because if $a_n$ is a cube then $a_n+(-x_1)^3=y_1^{2k}$ and if $a_n$ is a $p$-th power for any odd prime $p\ne{3}$ dividing $d$ then $a_n+(-x_2)^p=y_2^2$ which are not possible by Lemma \ref{b}.
		
Let $d=2^r$ with $r\geq 2$. Then $f(z)=z^{2^r}+\frac{1}{c}$ and  $a_{n}=a_{n-1}^{2^r}+c^{(2^r)^{n-1}-1},\, n\geq{3}$. We need to show for each $n\geq 3$ that $a_{n}$ is not a square in $\mathbb{Z}$. 
If $a_n=y^2$ for some $y$, then taking $k=(2^r)^{n-1}-1>6$ odd, we get an equation $y^2+(-c)^k=(a_n^{2^{r-2}})^4$ which has no primitive solution by Lemma \ref{b}. 
			
Let $d=2^r\cdot 3^s$ with $r, s\geq 1$. Then $6|d$ and hence $a_{n}=a_{n-1}^d+c^{d^{n-1}-1}=x^6+y^k$ for some integers $x, y, k$ with odd $k\geq 5$. Considering equations 
$a_n+(-y)^k=x^6$ if $a_n$ is a square and $a_n+(-y)^k=x^6$ if $a_n$ is a cube, we get a contradiction by Lemma \ref{b}.  Hence $a_n$ is neither a square nor a cube for $n\geq 3$. 

Let $d=2^r\cdot5^s\cdot7^t$ with $r\geq 1,s,t\geq{0}$ and $d\equiv1\pmod3$. We assume that either $s>0$ or $t>0$ since $d=2^r$ is already considered. If $a_n=z^p$ for a prime $p|d$, we have 
$$a_n=a_{n-1}^d+c^{d^{n-1}-1}=\begin{cases}
x_1^{10}+y^3\implies z^2+(-y)^3=x^{10}_1  &{\rm if} \ s>0 \ {\rm and} \ a_n=z^2\\
x_3^7+y^3\implies z^2+(-y)^3=x_3^7 & {\rm if} \ s=0 \ {\rm and} \ a_n=z^2\\
x_2^5+y^3\implies z^5+(-x_2)^5=y^3 & {\rm if} \ s>0 \ {\rm and} \ a_n=z^5\\
x_3^7+y^3\implies x^7_3+(-z)^7=x_3^7 & {\rm if} \ t>0 \ {\rm and} \ a_n=z^7.
\end{cases}$$ 
We get a contradiction by Lemma \ref{b}.    

 Let $d\equiv 4\pmod{12}$ with $d>4$. Then $3|(d-1)$ and $4|d$ and we have 
 $$a_n=a_{n-1}^d+c^{d^{n-1}-1}=x^4_1+y^k_1=x^p_2+y^3_2$$ 
 some integers $x_1, x_2, y_1, y_2, k, p$ with odd $k\geq 6$ and  for any odd prime $p|d$. 
 Considering equations $a_n+(-y_1)^k=x^4_1$ if $a_n$ is a square and $a_n+(-x)^p=y^3_2$ if $a_n$ is a $p-$th power for an odd prime $p|d$, we get a contradiction by Lemma \ref{b}.  
 Hence $a_n$ is not a $p-$the power for any prime divisor of $d$. This proves Theorem \ref{main1}.

\subsection{{\bf Proof of Theorem \ref{main2}:} } 
 We may assume, by Theorem \ref{main1}, that $d$ is an even integer which is not considered in Theorem \ref{main1}. As stated after the proof of Lemma \ref{7} we need aWe already noted , that it is enough to show that for $n\geq 3$, $a_{n}$ is not a $p$-th power in $\mathbb{Z}$ for any prime $p$ dividing $d$. 
Let $p|d$ and suppose that $a_n=z^p$ for some integer $z$.  Then 
$$z^p=a_{n}=a_{n-1}^d +c^{d^{n-1}-1}=x^p+y^k$$ 
for some integers $x, y, k$ with odd $k\geq 5$. 
 
 From now on, we assume explict abc Conjecture \ref{exp-abc} . By Proposition \ref{abc-Beal}, $z^p+(-x)^p=y^k$ has no primitive solutions when $p$ is odd. Hence we assume that $p=2$ and therefore 
we have $z^2=a^d_{n-1}+c^{d^{n-1}-1}$.  It suffices to show that this equation is not possible.  We have $c<z^{\frac{2}{d^{n-1}-1}}$ and $a_{n-1}<z^{\frac{2}{d}}$ so that 
$$N=N(z^2\cdot a^d_{n-1}\cdot c^{d^{n-1}-1})=N(z\cdot a_{n-1}\cdot c)<z^{1+\frac{2}{d} +\frac{2}{d^{(n-1)}-1}}$$ where 
$N(r)$ is the radical of $r$.  By Theorem \ref{abc-LS} applied to $a^d_{n-1}+c^{d^{n-1}-1}=z^2$, we obtain 
 $z^2<N^{\frac{7}{4}}$ implying  
\begin{align*}
(z^2)^{\frac{4}{7}}<z^{1+\frac{2}{d} +\frac{2}{d^{n-1}-1}} \implies 
\frac{8}{7}<1+\frac{2}{d}+\frac{2}{d^{n-1}-1} \implies \frac{1}{14}<\frac{1}{d}+\frac{1}{d^{2}-1}.
  \end{align*}
since $n\geq 3$. Clearly $d>14$ so that $d^2-1>2d$ and hence 
 $$ \frac{1}{14}<\frac{1}{d}+\frac{1}{2d}=\frac{3}{2d} \implies d<21.$$
By Theorem \ref{main1}, we need to consider $d=14$ where we have the equation $x^2=a_{n}=a^{14}_{n-1}+c^{14^{n-1}-1}$.   
For $n$ odd, considering the equation  $x^2+(-y)^3=(a^2_{n-1})^7$ where 
$y^3=c^{14^{n-1}-1}$, we get a contradiction by Lemma \ref{b}. Hence we assume that $n$ is even. In particular $n\geq 4$.   
Also  $c=\pm 1$ is not possible by Lemma \ref{a} since $a_{n-1}\neq 0, \pm 1$.  Thus $|c|\geq 2$ and $n\geq 4$. We apply Theorem \ref{abc-LS}  to 
$a^{14}_{n-1}+c^{14^{n-1}-1}=z^2$. If  $$N=N(z^2\cdot a^{14}_{n-1}\cdot c^{14^{n-1}-1})<\exp(204.75),$$  
 we obtain 
 $$2^{14^3-1}\leq |c|^{14^{n-1}-1}\leq N^{\frac{7}{4}}\leq \exp(204.75\cdot \frac{7}{4})$$
which is a contradiction. Hence $N\geq \exp(204.75)$. Taking $\epsilon=\frac{7}{12}$ in Theorem \ref{abc-LS} again, we obtain $z^2<N^{1+\frac{7}{12}}$ implying 
\begin{align*}
 (z^2)^{\frac{12}{19}}<N\leq z^{1+\frac{2}{14} +\frac{2}{14^{n-1}-1}} \implies \frac{24}{19}<1+\frac{1}{7}+\frac{2}{14^{3}-1}<1+\frac{1}{7}+\frac{2}{19}
 \end{align*}
 since $n\geq 4$. This is a contradiction again. This proves Theorem \ref{main2}. 

	\bibliographystyle{plain}
	
	\end{document}